\documentclass[12pt]{amsart}

\DeclareMathAlphabet{\mathscr}{OT1}{pzc}{m}{it}

\usepackage{enumitem} 
\usepackage{braket}
\usepackage[margin=1.2in]{geometry}
\usepackage{hyperref}

\setlist{itemsep=3pt}

\newtheorem{prop}{Proposition}
\newtheorem{theo}[prop]{Theorem}
\newtheorem{lemm}[prop]{Lemma}

\newtheorem{claim}{Claim}

\newtheorem*{theo-lowdim}{Theorem \ref{theo:weakly-st-low-dim-est}$^{\prime}$}
\newtheorem*{theo-highdim}{Theorem \ref{theo:main-weakly-st-sheeting}$^{\prime}$}

\theoremstyle{definition}
\newtheorem{rema}{Remark}

\newtheorem*{theorem*}{Theorem}

\newcommand{\NN}{\mathbb{N}}

\newcommand{\RR}{\mathbb{R}}
\renewcommand{\SS}{\mathbb{S}}

\newcommand{\cC}{\mathcal C}
\newcommand{\cH}{\mathcal H}

\DeclareMathOperator{\supp}{spt}

\DeclareMathOperator{\Ric}{Ric}
\DeclareMathOperator{\sing}{sing}
\DeclareMathOperator{\Index}{index}

\title[CMC Curvature estimates] {Curvature estimates and sheeting theorems for weakly stable CMC hypersurfaces}

\author{Costante Bellettini} 
\address{Department of Mathematics\\
University College London\\
London WC1E 6BT, United Kingdom}
\email{c.bellettini@ucl.ac.uk}
\author{Otis Chodosh}
\address{Department of Mathematics\\Princeton
University\\Princeton, NJ 08544}
\address{School of Mathematics\\Institute for Advanced Study\\Princeton, NJ 08540}
\email{ochodosh@math.princeton.edu}
\author{Neshan Wickramasekera}
\address{DPMMS \\
University of Cambridge \\
Cambridge CB3 0WB, United Kingdom}
\email{ngw24@dpmms.cam.ac.uk}

\begin{document}

\begin{abstract}
Weakly stable constant mean curvature (CMC) hypersurfaces are stable critical points of the area functional with respect to volume preserving deformations. We establish a pointwise curvature estimate (in the non-singular dimensions) and a sheeting theorem (in all dimensions) for weakly stable CMC hypersurfaces, giving an effective version of the compactness theorem for weakly stable CMC hypersurfaces established in the recent work of the first- and third-named authors \cite{BW:stableCMC}. Our results generalize the curvature estimate and the sheeting theorem proven respectively by Schoen--Simon--Yau and Schoen--Simon for strongly stable hypersurfaces. 
\end{abstract}

\maketitle

\section{Introduction} 

In the recent work \cite{BW:stableCMC}, a regularity and compactness theory has been developed (in a varifold setting) for weakly stable constant-mean-curvature (CMC) hypersurfaces. The question of whether there is an effective version of the compactness theorem of \cite{BW:stableCMC}, i.e.\ whether \emph{weakly} stable CMC hypersurfaces must satisfy a uniform local curvature estimate under appropriate hypotheses, arises naturally from that work. Here we settle this  question by proving, for such hypersurfaces satisfying uniform mass and mean curvature bounds, a pointwise curvature estimate in the non-singular dimensions (i.e.\ in dimensions $\leq 6$) and a sheeting theorem (i.e.\ a pointwise curvature estimate subject to the additional hypothesis that the hypersurface is weakly close to a hyperplane) in all dimensions.  Our results generalize the foundational works of Schoen--Simon--Yau \cite{SSY} that established a pointwise curvature estimate for \emph{strongly} stable minimal hypersurfaces in low dimensions and of Schoen--Simon \cite{SchoenSimon} that established a sheeting theorem in all dimensions for a class of strongly stable hypersurfaces (including CMC hypersurfaces) subject to an a priori smallness hypothesis  on the singular set.

Recall that a smooth immersion $x \, : \, \Sigma \to {\mathbb R}^{n+1}$  has constant mean curvature if and only if every compact portion
$\Sigma_{1} \subset \Sigma$ is stationary with respect to the hypersurface area functional $\mathscr{a}(\Sigma_{1})$ for volume-preserving deformations. This condition is equivalent to the fact that for some constant $H$, every compact portion  $\Sigma_{1} \subset \Sigma$ is stationary with respect to the functional 
$$J (\Sigma_{1}) = \mathscr{a}(\Sigma_{1}) + H \mathscr{vol}(\Sigma_{1})$$ 
for arbitrary deformations, where 
$\mathscr{vol}(\Sigma_{1})$ is the enclosed volume functional (which can be expressed as $\mathscr{vol}(\Sigma_{1}) = \frac{1}{n+1}\int_{\Sigma_{1}} x \cdot \nu \, d\Sigma$ where $\nu$ is a continuous unit normal to $\Sigma$ and $d\Sigma$ is the volume element with respect to the metric induced by the immersion $x$); in this case, $H$ is the value of the scalar mean curvature of $\Sigma$ with respect to $\nu$. If $\Sigma$ has constant mean curvature, then for any given $\phi \in C^{\infty}_{c}(\Sigma)$ and relative to any smooth 1-parameter family of deformations of $\Sigma$ with initial velocity $\phi \nu$, the second variation of $\Sigma$ with respect to $J$ is given by the quadratic form $$\delta^{2}J(\phi,\phi) = \int_{\Sigma} |\nabla \phi|^{2} - |A_{\Sigma}|^{2} \phi^{2},$$ where  $A_{\Sigma}$ is the second fundamental form of $\Sigma$ and $\nabla$ is the gradient on $\Sigma$ (cf.\ \cite[Proposition 2.5]{BdoCE}).  
We say that the CMC hypersurface $\Sigma$ is weakly stable  if every compact portion $\Sigma_{1} \subset \Sigma$ is stable, i.e.\ has non-negative second variation, with respect to the area functional, or equivalently, with respect to $J$, for volume-preserving deformations. Weakly stable CMC hypersurfaces arise as stable critical points for the isoperimetric problem. The weak stability of $\Sigma$ is equivalent to the validity of the stability inequality 
\[
\int_\Sigma |A_{\Sigma}|^2 \phi^{2} \leq  \int_{\Sigma} |\nabla \phi|^2
\]
for any $\phi \in C^\infty_c(\Sigma)$ with $\int_\Sigma \phi =0$ (cf.\ \cite[Proposition 2.7]{BdoCE}), while strong stability of $\Sigma$ requires that this inequality holds for arbitrary $\phi \in C^{\infty}_{c}(\Sigma)$. 

The methods used in \cite{SSY,SchoenSimon} for strongly stable hypersurfaces involve the use of \emph{positive} test functions $\phi$ in the stability inequality, and since these never integrate to zero, it is not clear how to directly apply these methods in the setting of weak stability.
The strategy employed here is different: we take a geometric approach, combining the results of \cite{SSY,SchoenSimon} for strongly stable hypersurfaces with the fact that complete weakly stable minimal hypersurfaces have only one end, a result established by Cheng--Cheung--Zhou (\cite{ChengCheungZhou}) and generalized here (in a fairly straightforward manner) to allow the hypersurfaces to have a small singular set. This generalization is necessary for the sheeting theorem. A key difficulty in the proof of the sheeting theorem is to correctly ``localize'' the one-end result in order to transfer the ``flatness'' from large to small scales (see Remark \ref{rema:transfer}). This is handled by a careful blow-up procedure relying on the aforementioned regularity and compactness theorems in \cite{BW:stableCMC} for weakly stable CMC hypersurfaces and a rigidity theorem (Lemma~\ref{lemm:cone-small-curvature} below), due to Simons (\cite{Simons}), for minimal hypersurfaces of spheres.

Our main results are Theorem~\ref{theo:weakly-st-low-dim-est}, Theorem~\ref{theo:main-weakly-st-sheeting}, Theorem~\hyperref[theo:varifolds-low-dim-est]{$1^{\prime}$} and Theorem~\hyperref[theo:varifolds-main-weakly-st-sheeting]{$2^{\prime}$} below. Theorem~\ref{theo:weakly-st-low-dim-est} gives a pointwise curvature bound valid for mass bounded weakly stable CMC hypersurface of dimension $n$ with $3 \leq n \leq 6$ (that are assumed, in case  $3 \leq n \leq 5$, to be immersed,  or in case $n=6$, immersed without transverse intersections or immersed with a specific mass bound); Theorem \ref{theo:main-weakly-st-sheeting} establishes a sheeting result that holds in arbitrary dimensions for weakly stable CMC hypersurfaces satisfying an arbitrary uniform mass bound and  allowed, a priori, to contain a small set of ``genuine'' singularities away from which the hypersurfaces are assumed smoothly immersed without transverse intersections. By virtue of the regularity theory of \cite{BW:stableCMC}, the hypotheses of absence or smallness of the set of genuine singularities in Theorem 1 and Theorem 2 respectively can immediately be replaced by considerably weaker structural conditions. These stronger results, which hold in a varifold setting, are given as 
Theorem~\hyperref[theo:varifolds-low-dim-est]{$1^{\prime}$} and Theorem~\hyperref[theo:varifolds-main-weakly-st-sheeting]{$2^{\prime}$}.

It is interesting to note the following: Consider a CMC hypersurface $\Sigma$ immersed in $\RR^{n+1}$ with mean curvature $H$ (possibly equal to zero).  Recall that the \emph{Morse index} of $\Sigma$ is defined by setting $\Index(\Sigma)$ to be the maximum dimension of a linear subspace $W$ of $C^{\infty}_{c}(\Sigma)$ so that for any $\phi \in W\setminus\{0\}$, the second variation $\delta^{2}J(\phi, \phi) <0$, or equivalently, 
\[
\int_{\Sigma} |A_{\Sigma}|^{2}\phi^{2} > \int_{\Sigma} |\nabla \phi|^{2}.
\]
It is easy to see that if $\Sigma$ is weakly stable, then $\Index(\Sigma) \leq 1$. On the other hand, Theorems \ref{theo:weakly-st-low-dim-est} and \ref{theo:main-weakly-st-sheeting} below are \emph{false} if we replace ``$\Sigma$ is weakly stable'' with ``$\Sigma$ satisfies $\Index(\Sigma) \leq 1$.'' This can be seen by considering rescalings of the higher-dimensional catenoid (the unique non-flat rotationally symmetric minimal hypersurface in $\RR^{n+1}$) which converge weakly to a hyperplane with multiplicity two, but do not have bounded curvature (or satisfy the conclusion of the sheeting theorem). In the context of the results below, the crucial difference between ``weakly stable'' and ``$\Index(\Sigma)\leq 1$'' is that weakly stable surfaces cannot have two ends (cf.\ Appendix \ref{appendixCCZ}) while index one surfaces can (e.g., the catenoid).

\subsection{Results for hypersurfaces with small singular set}\label{small-sing} 
In the non-singular dimensions (i.e.~in dimensions $\leq 6$), we have the following curvature estimates. 
\begin{theo}\label{theo:weakly-st-low-dim-est}
For each $H_0 >0$ and $\Lambda \geq 1$, there exists $C = C(H_{0}, \Lambda)$ such that the following holds: Let $3 \leq n \leq 6$ and let $\Sigma \subset B_{R}(0)\subset \RR^{n+1}$ be a smooth immersed hypersurface with $(\overline{\Sigma} \setminus \Sigma) \cap B_{R}(0) = \emptyset$, $\cH^{n}(\Sigma) \leq \Lambda R^{n}$ and with constant scalar mean curvature $H$ such that $|H| \leq H_{0}R^{-1}.$  Assume that $\Sigma$ is weakly stable as a CMC immersion. For $n=6$ suppose additionally  either that $\Sigma$ contains no point where $\Sigma$ intersects itself transversely (or equivalently, by the maximum principle, for each point $p \in \Sigma$ where $\Sigma$ is not embedded, there is $\rho >0$ such that $\Sigma \cap B^{n+1}_\rho(p)$ is the union of two embedded smooth CMC hypersurfaces intersecting only tangentially), or that $\Lambda = 3-\delta$ for some $\delta \in (0, 1)$. 

Then 
\[
\sup_{x \in \Sigma \cap B_{R/2}(0)} |A_{\Sigma}|(x) \leq C R^{-1},
\]
where $A_{\Sigma}$ denotes the second fundamental form of $\Sigma$.
\end{theo}

We note that when $n=2$ (cf.\ \cite{Ye,EichmairMetzger}) stronger estimates are available---i.e., without the bounded area assumption---as consequences of the strong Bernstein type theorems available \cite{Barbosa-doCarmo,BdoCE,Palmer,DaSilveira,LopezRos}. As such, we will not consider this case here. 

\begin{rema}
In case $n=6$, the reason for the additional restrictions in Theorem~\ref{theo:weakly-st-low-dim-est} (that either $\Sigma$ has no transverse points or $\Lambda = 3-\delta$) is that it is not known if a pointwise curvature estimate holds for 6-dimensional immersed strongly stable minimal hypersurfaces satisfying an arbitrary mass bound; such an estimate is only known to hold if the minimal hypersurface is either embedded (\cite{SchoenSimon}) or is immersed and satisfies a  
mass bound corresponding to $\Lambda = 3 - \delta$ for some $\delta \in (0, 1)$ (\cite{Wic08}). See Proposition~\ref{prop:bernstein-low-dim} below.  
\end{rema}

In all dimensions, we have the following sheeting theorem. 

\begin{theo}\label{theo:main-weakly-st-sheeting}
Let $\Lambda, H_0 >0$ and $n\geq 3.$ Suppose that $\Sigma^{n}\subset B_{R}(0)\subset \RR^{n+1}$ is an immersed hypersurface with $\cH^{n}(\Sigma) \leq \Lambda R^{n}$, with constant scalar mean curvature $H$ such that $|H| \leq H_{0}R^{-1}$ and with $\mathcal{H}^{n-7+\alpha}((\overline{\Sigma}\setminus \Sigma) \cap B_{R}(0))=0$ for all $\alpha >0$ (in other words, $\Sigma$ may have a co-dimension $7$ singular set). Suppose that 
$\Sigma$ contains no point where $\Sigma$ intersects 
itself transversely (or equivalently, by the maximum principle, for each $p \in \Sigma$ where $\Sigma$ is not embedded there is $\rho >0$ such that $\Sigma \cap B^{n+1}_\rho(p)$ is the union of exactly two embedded smooth CMC hypersurfaces intersecting only tangentially), and that $\Sigma$ is weakly stable as a CMC immersion. There exists $\delta_{0}=\delta_{0}(n,H_{0},\Lambda)$ and $C=C(n,H_{0},\Lambda)$ so that if additionally 
\[
\Sigma \subset \{|x^{n+1}| \leq \delta_{0}R\}
\]
then $\overline{\Sigma}\cap B_{R/2}(0)$ separates into the union of the graphs of functions $u_{1} \leq \dots \leq u_{k}$ defined on $B^{n}_{R/2}(0) := B_{R/2}(0)\cap \{x^{n+1}=0\}$ satisfying 
\[
\sup_{B^{n}_{R/2}(0)} \left( |Du_{i}| + R |D^{2}u_{i}|\right) \leq C\delta_{0} 
\]
for $i=1,\dots,k;$ moreover, each $u_i$ is separately a smooth CMC graph. 
\end{theo}

{
\begin{rema}
The constants in Theorems \ref{theo:weakly-st-low-dim-est} and \ref{theo:main-weakly-st-sheeting} depend on an upper bound for the mean curvature $H_{0}$. This cannot be removed; indeed, consider the hypersphere $\Sigma = \partial B_{r}(0)$, which is a weakly stable CMC embedding. Note that as $r\to0$, the curvature of $\Sigma$ blows up (in spite of the fact that $\Sigma$ is eventually contained in any slab). 
\end{rema}
}
\subsection{Results for varifolds} In view of \cite[Theorem 2.1]{BW:stableCMC}, Theorems \ref{theo:weakly-st-low-dim-est} and \ref{theo:main-weakly-st-sheeting} above imply the following stronger results for integral varifolds. We refer to \cite[Section 2.1]{BW:stableCMC} for precise definitions. Here we recall, slightly imprecisely, that: 
\begin{itemize}
 \item  a \textit{classical singularity} of an integral varifold $V$ is a point $p$ such that, in a neighbourhood of $p$, ${\rm spt} \, \|V\|$ (where $\|V\|$ denotes the weight measure associated with $V$) is given by the union of three or more embedded $C^{1,\alpha}$ hypersurfaces-with-boundary that intersect pairwise only along their common boundary $L$ containing $p$ and such that at least two of the hypersurfaces-with-boundary meet transversely along $L$;
 \item  a (two-fold) \textit{touching singularity} of an integral varifold $V$ is a point $p \in {\rm spt} \, \|V\|$ such that ${\rm spt} \, \|V\|$ is not embedded at $p$ and in a neighborhood of $p$, the ${\rm spt} \, \|V\|$ is given by the union of exactly two $C^{1,\alpha}$ embedded hypersurfaces with only tangential intersection;
 \item (see \cite{Simon:GMT} for details) the first variation of an integral varifold $V$ is a continuous linear functional on $C^1_c$ ambient vector fields and it represents the rate of change of the varifold's weight measure (area functional) computed along ambient deformations induced by the chosen vector field; when the first variation is a Radon measure (i.e.~it extends to a continuous linear functional on $C^0_c$ vector fields) the varifold is said to have locally bounded first variation; when, in addition, this Radon measure is absolutely continuous with respect to the weight measure $\|V\|$, and its Radon--Nikodym derivative (called generalized mean curvature of $V$) is in $L^p(\|V\|)$, the first variation of $V$ is said be locally summable to the exponent $p$ (with respect to the weight measure $\|V\|$). By the fundamental regularity theory of Allard, the class of integral $n$-varifolds $V$ with \textit{first variation locally summable to an exponent $p>n$} is compact in the varifold topology under uniform mass and $L^{p}$ mean curvature bounds, and such a $V$ enjoys an embryonic regularity property: there exists a dense open subset of ${\rm spt} \, \|V\|$ in which $\text{spt}\,\|V\|$ is $C^{1,\alpha}$ embedded, with $\alpha=1-\frac{n}{p}$ if $n < p < \infty$ and $\alpha \in (0, 1)$ arbitrary of $p = \infty$ (see \cite{Allard}). 
\end{itemize}

{In low dimensions, we have the following curvature estimates:}

\begin{theo-lowdim}\label{theo:varifolds-low-dim-est}
Let $\Lambda, H_0 >0$. For $3 \leq n \leq 6$, suppose that $V \in IV_n(B_{R}(0))$ is an integral varifold with $\|V\|(B_R(0)) \leq \Lambda R^n$. Assume that the following hypotheses hold:
\begin{enumerate}
 \item the first variation of $V$ is locally summable to an exponent $p>n$ (with respect to the weight measure $\|V\|$);
 \item $V$ has no classical singularities;
 \item whenever $p$ is a (two-fold) touching singularity there exists $\rho>0$ such that $${\mathcal{H}}^{n}\left(\{y \in \supp{\|V\|} \cap B_\rho(p) : \Theta(\|V\|,y) = \Theta(\|V\|,p)\}\right)=0,$$
 where $\Theta$ is the density;
 \item the $C^1$ embedded part of $\supp{\|V\|}$ (non-empty by Allard's regularity theorem) has generalized mean curvature $h$ with $|h|=H$ for a constant $H\leq H_0$ (see \cite{BW:stableCMC} for the variational formulation of this assumption, which makes sense for a $C^1$ hypersurface and leads to its $C^2$ regularity by standard elliptic regularity theory);
 \item the $C^2$ immersed part of $\supp{\|V\|}$ (which is a CMC immersion in view of (4)) is weakly stable, i.e.\ stable for the area measure under volume-preserving variations.
\end{enumerate}

Then $\Sigma=\supp{\|V\|} \cap B_{R}(0)$ is a smooth immersion and there is $C=C(H_{0},\Lambda)$ so that 
\[
\sup_{x \in \Sigma \cap B_{R/2}(0)} |A_{\Sigma}|(x) \leq C R^{-1},
\]
where $A_{\Sigma}$ denotes the second fundamental form of $\Sigma$.
\end{theo-lowdim}

We also have the following sheeting theorem in all dimensions:
\begin{theo-highdim}\label{theo:varifolds-main-weakly-st-sheeting}
Let $\Lambda, H_0 >0$. For any $n\geq 3$ suppose that $V \in IV_n(B^{n+1}_{R}(0))$ is an integral varifold with $\|V\|(B^{n+1}_R(0)) \leq \Lambda R^n$. Assume that the following hypotheses hold:
\begin{enumerate}
 \item the first variation of $V$ is locally summable to an exponent $p>n$ (with respect to the weight measure $\|V\|$); 
 \item $V$ has no classical singularities;
 \item whenever $p$ is a (two-fold) touching singularity there exists $\rho>0$ such that $${\mathcal{H}}^{n}\left(\{y \in \supp{\|V\|} \cap B_\rho^{n+1}(p) : \Theta(\|V\|,y) = \Theta(\|V\|,p)\}\right)=0,$$
 where $\Theta$ stands for the density;
 \item the $C^1$ embedded part of $\supp{\|V\|}$ (non-empty by Allard's regularity theorem) has generalized mean curvature $h$ with $|h|=H$ for a constant $H\leq H_0$ (see \cite{BW:stableCMC} for the variational formulation of this assumption, which makes sense for a $C^1$ hypersurface and leads to its $C^2$ regularity by standard elliptic methods);
 \item the $C^2$ immersed part of $\supp{\|V\|}$ (which is a CMC immersion in view of (4)) is weakly stable, i.e.\ stable for the area measure under volume-preserving variations.
\end{enumerate}

There exists $\delta_{0}=\delta_{0}(n,H_{0},\Lambda)$ so that if additionally 
\[
\supp{\|V\|} \subset \{|x^{n+1}| \leq \delta_{0}R\}
\]
then $\supp{\|V\|} \cap B_{R/2}(0)$ separates into the union of the graphs of functions $u_{1} \leq \dots \leq u_{k}$ defined on $B^{n}_{R/2}(0) := B_{R/2}(0)\cap \{x^{n+1}=0\}$ satisfying 
\[
\sup_{B^{n}_{R/2}(0)} \left( |Du_{i}| + R |D^{2}u_{i}|\right) \leq \delta_{0} 
\]
for $i=1,\dots,k;$ moreover, each $u_i$ is separately a smooth CMC graph.  
\end{theo-highdim}

\begin{rema}
The extension of the theorems above to the case of an ambient Riemannian manifold follows the same arguments, employing the result in \cite{BW:stableCMC-future}. 
\end{rema}

\begin{rema}
Note that Theorems \ref{theo:weakly-st-low-dim-est}, \ref{theo:main-weakly-st-sheeting}, \hyperref[theo:varifolds-low-dim-est]{$1^{\prime}$} and \hyperref[theo:varifolds-main-weakly-st-sheeting]{$2^{\prime}$} hold in particular for $H=0$; in this case, the vanishing of the mean curvature prevents touching singularities, therefore assumption (3) in Theorems \hyperref[theo:varifolds-low-dim-est]{$1^{\prime}$} and  \hyperref[theo:varifolds-main-weakly-st-sheeting]{$2^{\prime}$} is redundant. For $H=0$ our results generalize the works of Schoen--Simon--Yau \cite{SSY}, Schoen--Simon \cite{SchoenSimon} and the third author \cite[Theorem 3.3]{W:annals2014} from strong to weak stability. 
\end{rema}

\begin{rema}
The conclusions of Theorems \ref{theo:main-weakly-st-sheeting} and \hyperref[theo:varifolds-main-weakly-st-sheeting]{$2^{\prime}$} clearly fail (even for strongly stable minimal hypersurfaces) for $n\geq 7$ without any flatness assumption, by the construction of Hardt--Simon \cite{HardtSimon:foliation}.
We also note that singularities do occur in stable CMC hypersurfaces (with $H \neq 0$) of dimension $\geq 7$, as shown by a recent construction of Irving (\cite{Irving})  modifying the earlier work of 
Caffarelli--Hardt--Simon (cf.\ \cite{CHS}). 
\end{rema}

\subsection{A remark on bounded index minimal surfaces} 
\label{boundedindexminimal}
The discussion in the paragraph preceding Section~\ref{small-sing} notwithstanding, the techniques developed in this paper are relevant for the study of bounded index minimal surfaces in Riemannian $(n+1)$-manifold for $n\geq 7$ (i.e., in the singular dimensions). For example, if $\Sigma^{n}\subset B_{R}(0)\subset \RR^{n+1}$ is a minimal surface with $\Index(\Sigma) \leq 1$, $\cH^{n}(\Sigma) \leq \Lambda R^{n}$, and $\Sigma\subset \{|x^{n+1}| \leq \delta_{0}R\}$, then by a straightforward application of the Schoen--Simon sheeting theorem \cite{SchoenSimon}, $\Sigma$ splits into smooth sheets away from a given point. The argument used to prove Proposition \ref{prop:curv-est-sheeting-all-dim} extends to this situation to conclude that the sheets are connected by a small region that is close (depending on $\delta_{0}$) to an index one minimal hypersurface in $\RR^{n+1}$ (with small singular set), \emph{having regular ends}. This last condition is the non-trivial conclusion; it follows from the argument in Proposition \ref{prop:curv-est-sheeting-all-dim}, transferring flatness from large scales to small scales (see Remark \ref{rema:transfer}). Using the arguments in \cite{CKM} (cf.\ \cite{BS:index}), similar statements hold for $\Index(\Sigma)\leq I_{0}$. See also \cite{Tysk}.

\subsection{Outline of the paper} Theorem \ref{theo:weakly-st-low-dim-est} will be proved in Section \ref{proofof1}, building on the Bernstein-type result given in Proposition \ref{prop:bernstein-low-dim} below (Section \ref{Bernsteintypetheorems}). Theorem \ref{theo:main-weakly-st-sheeting} will be proved in Section \ref{proofof2}, building on a different Bernstein-type result (Proposition \ref{prop:bernstein-all-dim} in Section \ref{Bernsteintypetheorems}). The proofs of both Bernstein-type results rely on a global result for weakly stable minimal hypersurfaces, namely the fact that they must be one-ended. This is proved in \cite{ChengCheungZhou} in the case of \emph{smooth} embedded hypersurfaces; this result, recalled in Theorem \ref{theo:CCZ-smooth} of Appendix \ref{appendixCCZ}, is all that is actually needed for Theorem \ref{theo:weakly-st-low-dim-est}, together with a classical blow-up argument. For the proof of Theorem \ref{theo:main-weakly-st-sheeting}, we extend the one-ended conclusion to the situation where the hypersurface may have a codimension-$7$ singular set; this is done in Theorem \ref{theo:CCZ-non-smooth} of Appendix \ref{appendixCCZ}. The proof of Theorem \ref{theo:main-weakly-st-sheeting} also relies on a careful blow-up argument for which we need to use certain results from \cite{BW:stableCMC}, which we recall in Appendix \ref{appendixBW}.

\subsection{Acknowledgements} C.B. was supported by an EPSRC grant EP/S005641/1. O.C. was supported by an EPSRC grant EP/K00865X/1 as well as the Oswald Veblen Fund and NSF grants DMS-1638352 and DMS-1811059.  We are grateful to the referee for several useful suggestions. 

\section{Two Bernstein-type theorems}
\label{Bernsteintypetheorems}

We begin with the following Bernstein type result, which will yield Theorem \ref{theo:weakly-st-low-dim-est} when combined with a standard blow-up argument. We note that such a result holds for $n=2$ without the embededness or area growth assumptions, as discussed above. As a notational remark, we stress that we will always write $\nabla$ to denote the intrinsic gradient on a hypersurface, and will instead denote by $\nabla^{\RR^{n+1}}$ the ambient gradient.
\begin{prop}\label{prop:bernstein-low-dim}
For $3 \leq n \leq 6$, suppose that $\Sigma^{n}\subset \RR^{n+1}$ is a connected, weakly stable, immersed minimal hypersurface with no singularities and with $\cH^{n}(\Sigma\cap B_{R}) \leq \Lambda R^{n}$ for some constant $\Lambda \geq 1$ and all $R>0$. When $n=6$ assume either that $\Lambda = 3 - \delta$ for some $\delta >0$ or that $\Sigma$ is embedded. Then $\Sigma$ is a hyperplane. 
\end{prop}
\begin{proof}
We begin by showing that $\Sigma$ is (strongly) stable outside of a compact set. If all of $\Sigma$ is strongly stable, then by \cite{SSY,SchoenSimon} the proposition follows. If not, we may choose $R>0$ so that $\Sigma\cap B_{R}$ is unstable. If $\Sigma\setminus B_{2R}$ is unstable, then we may find functions $\varphi_{1} \in C^{\infty}_{c}(\Sigma\cap B_{R})$ and $\varphi_{2} \in C^{\infty}_{c}(\Sigma\setminus B_{2R})$ so that
\[
\int_{\Sigma} |A_{\Sigma}|^{2} \varphi_{i}^{2} > \int_{\Sigma} |\nabla \varphi_{i}|^{2} . 
\]
By weak stability, $\int_{\Sigma} \varphi_{i} \neq 0$ for $i=1, 2$. Choose $t \in {\mathbb R}$ so that
\[
\int_{\Sigma} \varphi_{1} + t \varphi_{2} = 0. 
\]
Because $\varphi_{1},\varphi_{2}$ have disjoint support, we find that 
\[
\int_{\Sigma} |A_{\Sigma}|^{2} (\varphi_{1}+t\varphi_{2})^{2} > \int_{\Sigma} |\nabla (\varphi_{1} + t \varphi_{2})|^{2}.
\]
This contradicts the weak stability of $\Sigma$. Thus, $\Sigma$ is stable outside of a compact set.

We first assume that $\Sigma$ is embedded. We will explain below the modifications for the cases $\Sigma$ immersed and $3\leq n\leq 5,$ or $\Sigma$ immersed, $n=6$ and $\Lambda = 3 -\delta$. In the embedded case, we first show that there exists an integer $m$ such that any tangent cone at infinity is a hyperplane with multiplicity $m$. 
\begin{claim}\label{claim:blowdown-low-dim}
There is $m \in \NN$ so that for any sequence $\lambda_{j}\to0$, a subsequence of $\Sigma_{j}:=\lambda_{j}\Sigma$ converges smoothly and graphically on any compact subset of $\RR^{n+1}\setminus\{0\}$ to a hyperplane of multiplicity $m$.
\end{claim}
\begin{proof}[Proof of the Claim]
By \cite[Theorem 3]{SchoenSimon} (for $n\leq 5$ the estimates in \cite{SSY} suffice) the magnitude of the second fundamental form decays as $\frac{1}{|y|}$ for $|y|\to \infty$, namely there exist $R_0>0$ and a constant $C>0$ such that $|A|(y) \leq \frac{C}{|y|}$ for $y \in \Sigma$ and $|y| \geq R_0$, where $|y|$ denotes the Euclidean norm of $y$ in $\RR^{n+1}$. {Therefore, there is a subsequence $\lambda_{j}\to 0$ (not relabeled) so that $\Sigma_{j}=\lambda_{j}\Sigma$ converges smoothly (possibly with multiplicity) on compact subsets of $\RR^{n+1}\setminus\{0\}$ to $\cC$, a smooth minimal surface in $\RR^{n+1}\setminus\{0\}$.} The smooth convergence implies that $\cC \setminus \{0\}$ is (strongly) stable: by \cite{Simons} and the dimensional restriction, $\cC$ is a flat hyperplane with some multiplicity $m \in \NN$. Finally, the fact that the multiplicity $m$ is independent of the sequence $(\lambda_{j})$ is an immediate consequence of the monotonicity formula. 
\end{proof}

The preceding claim implies that there exists $r_0$ such that, whenever $r>r_0$, the sphere $\partial B^{n+1}_r(0)$ intersects $\Sigma$ transversely: indeed, if that failed, we could produce a sequence of radii $r_i \to \infty$ where transversality fails but the corresponding sequence $\frac{1}{r_i} \Sigma$ would fail to converge graphically at some point on $\partial B^{n+1}_1(0)$. 

Let $r>r_0$. The transversality condition established amounts to the fact that the gradient of $h:\Sigma \setminus \overline{B}_{r_0}(0) \to \RR$, $h(x)=|x|$ (the ambient distance to the origin) is everywhere non-vanishing. By \cite[Theorem 3.1]{MilnorMorse} this implies that, for any $R>r$, $\Sigma \cap \left(B_R(0) \setminus B_r(0)\right)$ deformation retracts onto $\Sigma \cap \partial B_r(0)$. In particular, the number of connected components of $\Sigma \cap \left(B_R(0) \setminus B_r(0)\right)$ equals the number of connected components of $\Sigma \cap \partial B_r(0)$. Denoting with $D_1, \dots, D_N$ the connected components of $\Sigma \cap \partial B_r(0)$, we consider, for every $R>r$, $N$ disjoint open sets $A^R_1, \ldots, A^R_N$, each containing a single connected component of $\Sigma \cap \left(B_R(0) \setminus B_r(0)\right)$ and labelled so that $A^R_j$ contains $D_j$. Let $\tilde{A}_j=\cup_{R>r} A^R_j$: the open sets $\tilde{A}_j$ for $j=1, \dots, N$ are disjoint by construction and cover $\Sigma \setminus B_r(0)$, so the number of ends of $\Sigma$ is at least $N$.

The result of \cite{ChengCheungZhou} (see Theorem \ref{theo:CCZ-smooth} below) gives that $\Sigma$ is one-ended, i.e.~$N=1$, and so, for all $r>r_0$, $\Sigma \cap \partial B_r(0)$ is connected. On the other hand, $\SS^{n-1}$ is simply connected and, as such, does not admit a nontrivial connected cover. Therefore, recalling claim \ref{claim:blowdown-low-dim}, we conclude that $m=1$, or equivalently, that the density of $\Sigma$ at infinity is $1$. Hence by the monotonicity formula $\Sigma$ is a cone with density at the vertex (which is equal to the density at infinity) equal to 1. Since the density of $\Sigma$ at any other point is also 1, it follows again by the monotonicity formula that $\Sigma$ is translation invariant along every direction so it is a hyperplane. 

We now consider the case where $\varphi: \Sigma\to\RR^{n+1}$ is only assumed to be immersed and either $3\leq n\leq 5$ or $n=6$ and $\Lambda = 3-\delta$. In this case, we still have, by the local uniform mass bounds, that for any sequence $\lambda_{j} \to 0$, a subsequence of $\left(\lambda_{j}\right)_{\#} |\varphi(\Sigma)|$ converges as varifolds to a stationary cone ${\mathcal C}.$ By the locally uniform pointwise curvature bounds (given by  
\cite{SSY} for $3 \leq n \leq 5$ or by \cite{Wic08} for $n=6$ and $\Lambda = 3 - \delta$), it follows that ${\rm spt} \, \|{\mathcal C}\|$ is smoothly immersed away from the origin, and the convergence is smooth and graphical in compact subsets of ${\mathbb R}^{n+1} \setminus \{0\}$; moreover, since $\Sigma \setminus B_{2R}$ is stable, it also follows that  the stability inequality 
$\int |A_{\mathcal C}|^{2}\zeta^{2} \leq \int |\nabla \, \zeta|^{2}$ holds true for every $\zeta \in C^{1}_{c}({\rm spt} \, \|{\mathcal C}\| \setminus \{0\})$, i.e.\ that ${\rm spt} \, \|{\mathcal C}\| \setminus \{0\}$ is stable as an immersion.  (Indeed if $M_{j}$ is any sequence of immersed minimal hypersurfaces of an open set $U \subset {\mathbb R}^{n+1}$ with no singularities and with $\partial \, M_{j} \cap U = \emptyset$, and if $\limsup_{j \to \infty} \, {\mathcal H}^{n}(M_{j} \cap K) < \infty$ and $\limsup_{j \to \infty} \, \sup_{x \in M_{j} \cap K} \, |A_{M_{j}}(x)| < \infty$ for each compact $K \subset U$, then for any given compact set $K \subset U$, there is a fixed radius $\sigma  = \sigma(K)>0$ independent of $j$ such that (after passing to a subsequence without changing notation) for every $j$ and every $p \in M_{j} \cap K$, 
$M_{j} \cap B_{\sigma}(p)$ is the union of smooth embedded graphs with small gradient over some hyperplanes $P_{j,1}, \ldots, P_{j, N_{j}}$ passing through $p$ (with $\sum_{k=1}^{N_{j}}|P_{j, k}|$ equal to the tangent cone to $M_{j}$ at at $p$), where $N_{j} \leq N$ for some $N$ independent of $j$ and $p$; if $V$ is the varifold limit of $(M_{j})$, then for any $z \in {\rm spt} \, \|V\| \cap U$, choosing a sequence of points $z_{j} \in M_{j}$ with 
$z_{j} \to z$ and applying this fact  to $B_{\sigma}(z_{j}) \cap M_{j},$ we get, passing to a subsequence, that the hyperplanes $P_{j, k} \to P_{k}$ for $k =1, \ldots, Q$ for some $Q \leq N$, and so we can write 
$M_{j} \cap B_{\sigma/2}(z_{j})$ as a union of embedded minimal graphs over the fixed planes $P_{1}, \ldots, P_{Q}$ with small gradient. By the higher derivative estimates for solutions to uniformly elliptic equations, we then see that 
${\rm spt} \, \|V\| \cap  B_{\sigma/4}(z)$ is the union of smoothly embedded minimal graphs over $P_{1}, \ldots, P_{Q},$ i.e.\ that ${\rm spt} \, \|V\| \cap U$ is immersed, and that the convergence of 
$(M_{j})$ is smooth and graphical (via normal sections over ${\rm spt} \, \|V\| \cap U$) in any compact subset of $U$. From this, it is easy to verify that if $M_{j}$ are stable, i.e.\ if $\int_{M_{j}} |A_{M_{j}}|^{2} \zeta^{2} \leq \int_{M_{j}} |\nabla \, \zeta|^{2}$ for each $\zeta \in C^{1}_{c}(M_{j})$ then $\int |A_{{\rm spt} \, \|V\|}|^{2}\zeta^{2} \leq \int \nabla \, \zeta|^{2}$ for each $\zeta \in C^{1}_{c}({\rm spt} \, \|V\|)$.)

By Simons' theorem (\cite[Theorem 6.1.1]{Simons}; see the argument in \cite[Appendix B]{Simon:GMT} which is valid when the cone, as in our case, is immersed and stable as an immersion away from the origin), we conclude that ${\mathcal C} = \sum_{\ell = 1}^{M} m_{\ell}|L_{\ell}|$ for some hyperplanes $L_{1}, \ldots, L_{M}$ and positive integers $m_{1}, \ldots, m_{M}.$    
Arguing by contradiction (as in the embedded case), this shows that $\varphi$ is transverse to $\partial B_{r}^{n+1}(0)$ for all $r > r_{0}$ sufficiently large. Again, as in the embedded case, we thus find that the number of connected components of $\varphi^{-1}(B_{R}(0)\setminus B_{r}(0))$ is equal to the number of connected components of $\varphi^{-1}(\partial B_{r}(0))$ for any $R\geq r > r_{0}$. Because $\Sigma$ has only one end by Theorem \ref{theo:CCZ-smooth}, there is only one such component. This proves both that $\cC$ is supported on a single hyperplane, and that it has multiplicity one. Thus, $\Sigma$ is a flat hyperplane by the monotonicity formula. This completes the proof. 
\end{proof}

The proof of Theorem \ref{theo:weakly-st-low-dim-est} will be achieved by employing Proposition \ref{prop:bernstein-low-dim} and a standard blow-up argument (see Section \ref{proofof1}).
We now present a version of Proposition \ref{prop:bernstein-low-dim} that holds in all dimensions. This, in conjunction with the sheeting-away-from-a-point result for weakly stable CMC hypersurfaces from \cite{BW:stableCMC} (recalled in Appendix \ref{appendixBW}, Theorem \ref{theo:sheeting-away-from-pt} below), will imply Theorem \ref{theo:main-weakly-st-sheeting} using a less standard rescaling argument. We point out that, in the proof of the next proposition, we make use of the one-end result of \cite{ChengCheungZhou} for weakly stable CMC hypersurfaces, genralized here to allow a co-dimension $7$ singular set. This generalisation is given in Appendix \ref{appendixCCZ}, Theorem \ref{theo:CCZ-non-smooth}.

\begin{prop}\label{prop:bernstein-all-dim}
For $n \geq 3$, suppose that $V$ is a stationary integral $n$-varifold in $\RR^{n+1}$ with ${\rm spt} \, \|V\|$  a connected set, ${\rm sing} \, V \subset B_{1}(0)$ (so ${\rm spt} \, \|V\|$ is smooth in ${\mathbb R}^{n+1} \setminus B_{1}(0)$) and with ${\rm dim}_{\mathcal H} \, ({\rm sing} \, V) \leq n-7$. Assume that the regular part $\Sigma = {\rm reg} \, V$ ($= {\rm spt} \, \|V\| \setminus {\rm sing} \, V$) is weakly stable and that $V$ satisfies area growth $\Vert V\Vert (B_{R}(0)) \leq \Lambda R^{n}$ for some constant $\Lambda \geq 1$ and all $R >0$. Finally, assume that for some $\varepsilon >0$, $\Sigma$ satisfies 
\begin{equation}\label{eq:curv-est-bernstein}
|A_{\Sigma}|(x) |x| \leq \sqrt{n-1}-\varepsilon
\end{equation}
for $x \in \Sigma \setminus B_{1}$, where $|\cdot|$ denotes the length in $\RR^{n+1}$. Then $\supp \|V\|$ is a hyperplane. 
\end{prop}

\begin{proof}
We begin by proving that Claim \ref{claim:blowdown-low-dim} from Proposition \ref{prop:bernstein-low-dim} holds in this setting as well. For a sequence $\lambda_{j}\to 0$, we consider $V_{j}:=(\lambda_{j})_{\#} \, V$. Passing to a subsequence, $V_{j}$ converges to a cone $\cC$ in the sense of varifolds. Moreover, the assumed curvature estimates contained in \eqref{eq:curv-est-bernstein} imply that ${\rm spt} \, \|\cC\|\setminus\{0\}$ is smooth and $\Sigma_{j}$ converges smoothly to ${\rm spt} \, \|\cC\|$ (possibly with multiplicity) on compact subsets of $\RR^{n+1}\setminus\{0\}$ (here, we use the fact that the estimate \eqref{eq:curv-est-bernstein} is scale invariant). The curvature estimates pass to the limit, implying that $|A_{{\rm spt} \, \|\cC\|}|(x) |x| < \sqrt{n-1}$ for all $x \in {\rm spt} \,\|\cC\|\setminus\{0\}$. Appealing to Lemma \ref{lemm:cone-small-curvature} below, we find that $\cC$ is a flat hyperplane with some multiplicity $m \in {\mathbb N}$. This establishes Claim \ref{claim:blowdown-low-dim} in this setting (that the multiplicity $m$ is independent of the sequence follows again by monotonicity, as before). 

Thus, any tangent cone at infinity of $V$ is a multiplicity $m$ plane. By Theorem \ref{theo:CCZ-non-smooth}, applied to $V$, $\Sigma$ has exactly one end. Arguing as we did in the proof of Proposition \ref{prop:bernstein-low-dim}, we can use the graphical convergence on compacts sets in $\RR^{n+1}\setminus\{0\}$ (which follows from the curvature estimate (\ref{eq:curv-est-bernstein})) and the fact that $\SS^{n-1}$ does not admit any multiple cover, to obtain that, outside of $B_{1}$, $V$ must agree with the varifold given by $\Sigma$ with multiplicity $m$. Because the density at infinity of $V$ must be $m$, there must be equality in the monotonicity formula starting at any point in $\Sigma$ (which also has density $m$) which easily implies that the support of $V$ is a hyperplane. 
\end{proof}

\begin{lemm}\label{lemm:cone-small-curvature}
Suppose that $\cC$ is a $n$-dimensional minimal cone in $\RR^{n+1}$ that is smooth away from $0$ and satisfies $|A_{\cC}|(x)|x| < \sqrt{n-1}$. Then $\cC$ is a flat hyperplane. 
\end{lemm}
\begin{proof}
Note that $M : = \cC \cap {\mathbb S}^{n}$ is smooth. By the given curvature estimate, we have that $|A_{M}| < \sqrt{n-1}$. By \cite[Corollary 5.3.2]{Simons}, $M$ must be totally geodesic. This proves the assertion. 
\end{proof}

\begin{rema}\label{rema:transfer}
Observe that the Simons cone $\Sigma$ in $\RR^{8}$ is (strongly) stable and satisfies $|A_{\Sigma}|(x)|x| = \sqrt{n-1}$ for all $x \in \Sigma\setminus B_{1}$. As such, we see that the constant $\sqrt{n-1}-\varepsilon$ in \eqref{eq:curv-est-bernstein} is sharp in the sense that Proposition \ref{prop:bernstein-all-dim} fails with any larger constant. 

The importance of the size of the constant in a (scale invariant) curvature estimate of the form \eqref{eq:curv-est-bernstein} seems to have been first shown by White in \cite{Whi87}. This has been refined in \cite{MPR,CCE,CKM}. A key novelty contained in the present work is the combination of \eqref{eq:curv-est-bernstein} with Lemma \ref{lemm:cone-small-curvature} and with Theorem \ref{theo:CCZ-non-smooth}, allowing flatness to propagate from large to small scales. Furthermore, our work here seems to be the first use of such an estimate in a setting where a priori there could be singularities. 
\end{rema}

\section{Proof of Theorem \ref{theo:weakly-st-low-dim-est}}
\label{proofof1}

Because the hypothesis and conclusion are scale invariant, it suffices to take $R=1$. Assume the theorem is false. Then, there is $\Sigma_{j}$  in $B_{2}\subset \RR^{n+1}$, a sequence of embedded (when $\Sigma$ is immersed and $3\leq n\leq 5$, an identical argument will apply by considering instead rescalings and limits of the immersions) smooth weakly stable hypersurfaces with $|H| \leq H_{0}$ and $\cH^{n}(\Sigma_{j}) \leq \Lambda$, but so 
\[
\sup_{x \in \Sigma_{j}\cap B_{1/2}} |A_{\Sigma_{j}}|(x) \to \infty
\]
as $j\to\infty$. A standard blow-up argument (which we now recall) produces a surface which contradicts Proposition \ref{prop:bernstein-low-dim}. 

Choose $x_{j}\in\Sigma_{j}\cap B_{1/2}$ with $|A_{\Sigma_{j}}|(x_{j})\to\infty$. Without loss of generality, we may assume $x_j \to x_0$. Choose $\rho_{j}\to 0$ sufficiently slowly so that $\rho_{j}|A_{\Sigma_{j}}|(x_{j})\to\infty$. Find $y_{j}\in \Sigma_{j}\cap B_{\rho_{j}}(x_{j})$ maximizing 
\[
y\mapsto |A_{\Sigma_{j}}|(y) d(y,\partial B_{\rho_{j}}(x_{j})).
\]
Set $\sigma_{j} = d(y_{j},\partial B_{\rho_{j}}(x_{j}))$ and $\lambda_{j} = |A_{\Sigma_{j}}|(y_{j})$. Clearly $\sigma_j \leq \rho_j$ and $y_j \to x_0$, so that
\begin{equation}\label{eq:conseq-point-picking}
|A_{\Sigma_{j}}|(y) d(y,\partial B_{\rho_{j}}(x_{j})) \leq \sigma_{j}\lambda_{j} \text{ for } y\in \Sigma_j \cap B_{\sigma_j}(y_j).
\end{equation}
By the choice of $y_j$ we have $\sigma_j \lambda_j \geq \rho_j |A_{\Sigma_{j}}|(x_{j})$, which implies $\lambda_j:=|A_{\Sigma_{j}}|(y_{j}) \to \infty$ and $\lambda_j \sigma_j \to \infty$ as $j\to \infty$. We now define
\[
\tilde \Sigma_{j} = \lambda_{j}(\Sigma_{j}-y_{j}). 
\]
We claim that $\tilde\Sigma_{j}$ has bounded curvature on compact subsets of $\RR^{n+1}$. Indeed, for $x \in \tilde\Sigma_{j} \cap B_{\sigma_j \lambda_j}(0)$, scaling and \eqref{eq:conseq-point-picking} yield
\[
|A_{\tilde \Sigma_{j}}|(x) = \frac{1}{\lambda_{j}} |A_{\Sigma_{j}}|(y_{j} + \lambda_{j}^{-1}x) \leq \frac{\sigma_{j}}{\sigma_{j} - \lambda_{j}^{-1}|x|} \to 1
\]
for $|x|$ uniformly bounded. Note that $\tilde\Sigma_{j}$ has mean curvature $|H_{j}|\leq H_{0}/\lambda_{j} \to 0$. 

The monotonicity formula (see e.g.~\cite{Simon:GMT}) shows that $\cH^{n}(\tilde \Sigma_{j}\cap B_{R}) \leq \tilde \Lambda R^{n}$ for some constant $\tilde \Lambda=\tilde\Lambda(\Lambda,n,H_{0})$ independent of $j$. Then, by higher order elliptic estimates, $\tilde\Sigma_{j}$ converges (up to passing to a subsequence) smoothly (possibly with multiplicity) to a smooth, embedded, complete, weakly stable minimal hypersurface $\tilde\Sigma_{\infty}$ in $\RR^{n+1}$. 

Because $|A_{\tilde\Sigma_{j}}|(0) = 1$ for every $j$, we find that $|A_{\tilde\Sigma_{\infty}}|(0) = 1$, so $\tilde\Sigma_{\infty}$ is non-flat. This contradicts Proposition \ref{prop:bernstein-low-dim} (applied to $\tilde\Sigma_{\infty}$ with multiplicity one).

\section{Proof of Theorem \ref{theo:main-weakly-st-sheeting}}
\label{proofof2}

We begin by describing the setup of the proof of Theorem \ref{theo:main-weakly-st-sheeting}. By scaling we may take $R=10$. We consider a sequence of weakly stable hypersurfaces $\Sigma_{j}$ with mean curvature $|H| \leq H_{0}/10$ and $\cH^{n}(\Sigma_{j}) \leq \Lambda 10^{n}$. We assume that each $\Sigma_{j}$ has a singular set of co-dimension at least $7$ and that $\Sigma_{j}$ satisfies
\[
\Sigma_{j} \subset \{|x^{n+1}| \leq 10/j\} \cap B_{10}(0).
\]
It follows that $\Sigma_{j}$ converges to the flat disk $\{x^{n+1}=0\} \cap B_{10}(0)$  (smoothly away from a point by Theorem~\ref{theo:sheeting-away-from-pt}) with some positive integer multiplicity $k$, in the sense of varifolds. The final aim is to show that the conclusion of Theorem \ref{theo:main-weakly-st-sheeting} is valid for all sufficiently large $j$.

We will first establish the regularity assertion and the curvature estimate in Proposition~\ref{prop:curv-est-sheeting-all-dim} below; the proof of Theorem \ref{theo:main-weakly-st-sheeting} will then be completed at the end of the section. The curvature estimate of Proposition~\ref{prop:curv-est-sheeting-all-dim} will be a consequence of Proposition \ref{prop:bernstein-all-dim} and a blow-up argument. Its scale-breaking nature is reminiscent of the arguments in \cite{CKM}.  
\begin{prop}\label{prop:curv-est-sheeting-all-dim}
Fix $\eta>0$. Then, for $j$ sufficiently large, $\Sigma_{j} \cap B_{9}$ is smooth and there is $z_{j}\in B_{6}$ so that 
\[
|A_{\Sigma_{j}}|(x)|x-z_{j}| \leq \eta
\]
for all $x \in \Sigma_{j}\cap B_{9}$, where $|\cdot|$ stands for the length in $\RR^{n+1}$.
\end{prop}
We briefly explain the idea of the proof. The conclusion is non-trivial only when we are in the second alternative of the partial sheeting result from \cite{BW:stableCMC} that is recalled in Theorem \ref{theo:sheeting-away-from-pt}, Appendix \ref{appendixBW}. This second alternative gives that, away from a point, $\Sigma_{j}$ is converging smoothly (with sheeting) to a hyperplane with multiplicity. Thus, there is some $y$ and $\delta>0$ small so that the conclusion holds outside of $B_{\delta}(y)$. 

The strategy of the proof is to pick the \emph{smallest} ball $B_{\delta_{j}}(z_{j})$ so that the conclusion for $\Sigma_{j}$ holds outside of the ball. The claim will follow if we can prove that actually $\delta_{j}=0$, so we will assume that $\delta_{j}>0$. Rescale $\Sigma_{j}$ to $\hat\Sigma_{j}$ so that the ball $B_{\delta_{j}}(z_{j})$ becomes $B_{1}(0)$ (outside of which, the smoothness and \emph{scale invariant} curvature estimates hold). We can pass $\hat\Sigma_{j}$ to the limit, which inherits the curvature estimates (and smoothness) outside of $B_{1}(0)$. By Proposition \ref{prop:bernstein-all-dim}, the limit is a union of hyperplanes (note that here we have transferred the flatness estimates contained in the partial sheeting result to the smaller scale, as pointed out in Remark \ref{rema:transfer}). 
Now, the partial sheeting result (applied to $\hat\Sigma_{j}$) implies, as above, that the convergence of $\hat\Sigma_{j}$ to the limit occurs smoothly away from a single point. This contradicts our choice of $B_{\delta_{j}}(z_{j})$, since for $j$ large, we could take a smaller ball around the point where sheeting fails in the rescaled picture. This will contradict the assumption that $\delta_{j}>0$, and will complete the proof. 
\begin{proof}[Proof of Proposition \ref{prop:curv-est-sheeting-all-dim}]Clearly, it suffices to assume that $\eta < \sqrt{n-1}$. If the the first case of the conclusion of Theorem~\ref{theo:sheeting-away-from-pt} holds for every $\Sigma_{j}$ large enough, then the curvature estimate is true with $z_j=0$ (and the conclusion of Theorem \ref{theo:main-weakly-st-sheeting} is valid, so there is nothing further to prove). So we may assume (by the second case of the conclusion of Theorem~\ref{theo:sheeting-away-from-pt}) that there is a point $y \in B_{5}(0) \cap \{x^{n+1} = 0\}$ such that $\Sigma_{j}$ are sheeting away from $y$, i.e. for any $r>0$, $\Sigma_{j}\cap (B_{9}(0) \setminus B_{r}(y) )$ is smooth for $j$ sufficiently large and
\begin{equation}\label{eq:curv-away-from-pt}
\sup_{x\in\Sigma_{j}\cap (B_{9}(0) \setminus B_{r}(y) )} |A_{\Sigma_{j}}|(x) \to 0
\end{equation}
as $j\to\infty$. We will subsequently replace $\Sigma_{j}$ by $\Sigma_{j}\cap B_{9}(0)$ (to avoid any irrelevant issues with the behavior of $\Sigma_{j}$ near its boundary). 

For $z \in B_{6}(0)$, we define
\[
\delta(\Sigma_{j},z) : = \inf\left\{ r>0 : \begin{aligned} \Sigma_{j}^{r}:= \Sigma_{j}\setminus \overline{B_{r}(z)} \text{ is smooth \qquad} \\\text{and } |A_{\Sigma_{j}}|(x)|x-z| \leq \eta \text{ for all }x\in  \Sigma_{j}^{r} \end{aligned}  \  \right\}.
\]
Note that $\delta(\Sigma_{j},y) \to 0$ as $j\to \infty$, by the partial sheeting result discussed above. 

For every $j$ set $\delta_{j} : = \inf_{z \in B_{6}(0)} \delta(\Sigma_{j},z)$ and choose $z_{j,k}$ with $\delta(\Sigma_{j},z_{j,k}) \to \delta_{j}$ as $k\to\infty$. Passing to a subsequence, we may assume that $z_{j,k}\to z_{j} \in B_{6}(0).$ We claim that $\delta(\Sigma_{j},z_{j}) = \delta_{j}$. If not, there is $\epsilon > 0$ and $w \in \Sigma_{j} \setminus B_{\delta_{j}+2\epsilon}(z_{j})$ with either (i) $w \in \sing \Sigma_{j}$ or (ii) $|A_{\Sigma_{j}}|(w) |w-z_{j}| > \eta + 2\epsilon$. Note that $w \in \Sigma_{j}\setminus B_{\delta_{j} + \epsilon}(z_{j,k})$ for $k$ sufficiently large. Thus, in case (i), we find that, by the definition of $\delta(\cdot,\cdot)$, $\delta(\Sigma_{j},z_{j,k}) \geq \delta_{j} + \epsilon$ for all $k$ sufficiently large. This contradicts the choice of $z_{j,k}$. Similarly, in case (ii) we have that
\[
|A_{\Sigma_{j}}|(w)|w-z_{j,k}| > \eta + \epsilon,
\]
for $k$ sufficiently large, since $|w-z_{j,k}|\to |w-z_{j}|$ as $k\to\infty$. Again, this yields a contradiction, as before. 

Thus, we have arranged that $z_{j}$ minimizes $\delta(\Sigma_{j},\cdot)$. Since $\delta(\Sigma_{j}, y)  \to 0$, we also have that $\delta_{j} \to 0$ and consequently, it follows from the definition of $\delta_{j}$ and 
(\ref{eq:curv-away-from-pt}) that $z_{j} \to y.$ We claim that $\delta_{j} =0$ for all sufficiently large $j$. Arguing by contradiction, we assume (upon extracting a subsequence that we do not relabel) that $\delta_{j}>0$ for all $j$. Using this, we now perform the relevant blow-up argument. Define 
\[
\tilde\Sigma_{j} = \delta_{j}^{-1}(\Sigma_{j}-z_{j}).
\]
Note that as in the proof of Theorem \ref{theo:weakly-st-low-dim-est}, the monotonicity formula implies that $\cH^{n}(\tilde \Sigma_{j} \cap B_{R}(0)) \leq \tilde \Lambda R^{n}$ for some $\tilde \Lambda=\tilde \Lambda(\Lambda, n, H_0)$. Moreover, the choice of $\delta_{j}$ implies that $\tilde\Sigma_{j}\setminus B_{1}$ is smooth and satisfies
\begin{equation}\label{eq:sheeting-thm-proof-curv-est}
|A_{\tilde\Sigma_{j}}|(x)|x| \leq \eta
\end{equation}
for $x \in\tilde\Sigma_{j}\setminus B_{1}$. Note also that $|H_{\tilde{\Sigma}_j}|\leq \delta_j \frac{H_0}{10} \to 0$. The area bounds and weak stability imply, by the regularity and compactness theorems in \cite{BW:stableCMC} (recalled in Theorem~\ref{theo:BWregcomp}, Appendix \ref{appendixBW} below), that $\tilde \Sigma_{j}$ converge in the varifold sense to $\tilde V$, which is stationary, weakly stable, has smoothly embedded support outside of a co-dimension $7$ singular set, and satisfies $\Vert \tilde V\Vert(B_{R}(0)) \leq \tilde\Lambda R^{n}$. Furthermore, by the curvature estimates (\ref{eq:sheeting-thm-proof-curv-est}), the support of $\tilde V$ is a smooth hypersurface $\tilde \Sigma_{\infty}$ outside of $B_{1}(0)$ satisfying 
\[
|A_{\tilde\Sigma_{\infty}}|(x)|x| \leq \eta
\]
and the convergence is smooth on compact sets outside $B_{1}(0)$.
Thus, by Proposition \ref{prop:bernstein-all-dim}, each connected component of the support of $\tilde V$ is a hyperplane and so the support of $\tilde \|V\|$ is made up of finitely many parallel hyperplanes. 

Now, we again appeal to Theorem \ref{theo:sheeting-away-from-pt} to conclude that the convergence of $\tilde \Sigma_{j}$ to $\tilde V$ occurs smoothly (possibly with multiplicity) away from some fixed point $\tilde z \in 
\supp \|\tilde V\|$ (if the sheeting actually occurs everywhere, we simply set $\tilde z =0$). Note that the curvature estimates \eqref{eq:sheeting-thm-proof-curv-est} imply that $|\tilde z| \leq 1$.

Define $\hat z_{j} = z_{j} + \delta_{j}\tilde z$  and let $\hat\delta_{j} : = \delta(\Sigma_{j},\hat z_{j})$.  Since $\hat\delta_{j} \geq \delta_{j}/2$ (by the definition of $\delta_{j}$), there is $w_{j}\in\Sigma_{j} \setminus B_{\delta_{j}/2}(\hat z_{j})$ so that either (i) $w_{j}\in\sing \Sigma_{j}$, or (ii) $|A_{\Sigma_{j}}|(w_{j})|w_{j} - \hat z_{j}| > \eta$. We note that in either case we have that 
\begin{equation}\label{eq:sheeting-pf-scale-bettter-pt}
\liminf_{j\to\infty} \delta_{j}^{-1} |w_{j} - \hat z_{j}| = \infty. 
\end{equation}
(For if not, then defining $\tilde w_{j} = \delta_{j}^{-1}(w_{j}- z_{j})$, we find that, in the scale of $\tilde\Sigma_{j}$ discussed above,
\[
|\tilde z- \tilde w_{j}| = \delta_{j}^{-1} |(\hat z_{j} - z_{j})-(w_{j}-z_{j})| = \delta_{j}^{-1}|w_{j} - \hat z_{j}|
\]
is bounded above (after passing to a subsequence) and bounded below by $\frac 12$ 
(since $w_{j}\not \in B_{\delta_{j}/2}(\hat z_{j})$), but because $\tilde\Sigma_{j}$ sheets away from $\tilde z$, we find that either (i) or (ii) would be a contradiction.) 
Finally,  we define
\[
\check \Sigma_{j} := | w_{j}- \hat z_{j}|^{-1}(\Sigma_{j}-\hat z_{j})
\]
and set
\[
\check w_{j} : =  |w_{j}-\hat z_{j}|^{-1}(w_{j}-\hat z_{j}), \qquad \check z_{j} : = | w_{j} -\hat z_{j}|^{-1}(z_{j}-\hat z_{j}) = - \delta_{j} |w_{j} - \hat z_{j}|^{-1} \tilde z.
\]
Note that it follows from (\ref{eq:curv-away-from-pt}), (i) and (ii) that $w_{j} \to y$ and hence (since $z_{j} \to y$)  that $|w_{j} - \hat z_{j}| \to 0.$ We have already shown that outside of $B_{\delta_{j}}(z_{j})$, $\Sigma_{j}$ is smooth and satisfies \eqref{eq:sheeting-thm-proof-curv-est}. This implies that $\check\Sigma_{j}$ is smooth outside of $B_{\delta_{j}|w_{j}-\hat z_{j}|^{-1}}(\check z_{j})$ and additionally satisfies
\[
|A_{\check\Sigma_{j}}|(x) |x-\check z_{j}| \leq \eta.
\]
By \eqref{eq:sheeting-pf-scale-bettter-pt}, and recalling that $|\tilde z| \leq 1$, we have that $\delta_{j}|\hat z_{j}-w_{j}|^{-1} \to 0$ and $\check z_{j}\to 0$. As before, we may take the varifold limit $\check V$ of $\check \Sigma_{j}$, and the curvature estimates we have just established show that this limit occurs smoothly (possibly with multiplicity) outside of $B_{1/2}(0)$. The curvature estimates pass to the limit so by Proposition~\ref{prop:bernstein-all-dim} 
each connected component of ${\rm spt} \, \|\check V\|$ is a hyperplane.  Thus, since $|\check w_{j}|=1$, we find that $\check w_{j}\not \in \sing\check \Sigma_{j}$ and $|A_{\check\Sigma_{j}}|(\check w_{j})\to 0$, so
\[
|A_{\check\Sigma_{j}}|(\check w_{j})|\check w_{j}|\to 0.
\]
Rescaling to the original scale, this contradicts both (i) and (ii) above (concerning $w_{j}$).  This contradiction establishes that $\delta_{j} = 0$ for $j$ sufficiently large. 

We have now shown that $\Sigma_{j}$ is smooth away from $z_{j}$ and satisfies 
\[
|A_{\Sigma_{j}}|(x) |x-z_{j}|\leq \eta
\]
for all $x \in \Sigma_{j}\setminus \{z_{j}\}$. If $z_{j} \not \in \Sigma_{j}$ there is nothing further to show. Else, arguing as in the beginning of the proof of Proposition~\ref{prop:bernstein-low-dim}, we see that for fixed $j$ and any small ball $B_{\rho}(z_{j})$,  the hypersurface $\Sigma_{j}$ is strongly stable either in $B_{\rho}(z_{j})$ or in $B_{9} \setminus \overline{B_{\rho}(z_{j})};$ it follows from this fact that for each fixed $j$, there is $\rho>0$ such that $\Sigma_{j}$ is strongly stable in the annulus $B_{\rho}(z_{j}) \setminus \{z_{j}\}$, and hence in the ball $B_{\rho}(z_{j})$. Moreover, by the curvature estimates and Lemma \ref{lemm:cone-small-curvature}, any tangent cone $\Sigma_{j}$ at $z_{j}$ must be supported on a hyperplane. Thus, $\Sigma_{j}$ is smooth at $z_{j}$ by \cite[Theorems 3.1 and 3.3]{BW:stableCMC}. This completes the proof of Proposition \ref{prop:curv-est-sheeting-all-dim}. 
\end{proof}

\begin{proof}[Completion of the proof of Theorem \ref{theo:main-weakly-st-sheeting}]
As discussed in the beginning of this section, $\Sigma_{j}$ converge in the sense of varifolds to the hyperplane $\{x^{n+1}=0\}$ with multiplicity $k$. We claim that the curvature of $\Sigma_{j}$ is uniformly bounded on $\Sigma_{j}\cap B_{9}(0)$. Let $\lambda_{j} = \max_{\Sigma_j \cap B_9(0)} |A_{\Sigma_j}|$ and assume for a contradiction (upon extracting a subsequence that we do not relabel) that $\lambda_{j} \to \infty$ as $j \to \infty$. By applying Proposition \ref{prop:curv-est-sheeting-all-dim} iteratively we may find a further subsequence (not relabeled) so that $\Sigma_j$ is a smooth immersion in the whole ball $B_9(0)$ and there is $z_{j}\in B_{6}(0)$ and $\eta_{j}\to 0$ so that $\Sigma_{j}$ satisfies the curvature estimates 
\begin{equation}\label{eq:final-curv-est-sheeting-all-dim}
|A_{\Sigma_{j}}|(x)|x-z_{j}| \leq \eta_{j}
\end{equation} 
for all $x\in B_{9}(0)$. Note that since $z_j \in B_6(0)$, it follows from (\ref{eq:final-curv-est-sheeting-all-dim}) that $|A_{\Sigma_{j}}|$ is uniformly bounded in the annulus $B_9(0)\setminus \overline{B}_8(0)$ and therefore 
the maximum of $|A_{\Sigma_j}|$ in $\Sigma_{j} \cap B_{9}(0)$ is achieved at a point $y_j\in B_8(0)$. We set
\[
\tilde\Sigma_{j} = \lambda_{j}(\Sigma_{j}-y_{j}).
\]
By construction we have $|A_{\tilde\Sigma_{j}}|(0) = 1$ and that $|A_{\tilde\Sigma_{j}}|$ is uniformly bounded on compact subsets of ${\mathbb R}^{n+1}$, so $\tilde\Sigma_{j}$ converges smoothly on compact subsets of $\RR^{n+1}$ to a non-flat smooth hypersurface $\tilde\Sigma_{\infty}$. On the other hand, the estimate \eqref{eq:final-curv-est-sheeting-all-dim} is scale invariant, so for $\tilde z_{j} = \lambda_{j}(z_{j}-y_{j})$, we see that
\[
|A_{\tilde\Sigma_{j}}|(x) |x-\tilde z_{j}| \leq \eta_{j}.
\]
Considering $x=0$ here, we find that $\tilde z_{j}\to 0$, since $\eta_{j}\to 0$. Hence, passing this inequality to the limit, we find that
\[
|A_{\tilde\Sigma_{\infty}}|(x) |x| = 0,
\]
contrary to the fact that $\tilde\Sigma_{\infty}$ is non-flat. 

This implies that the curvature of $\Sigma_{j}$ (in the original scale, for the original sequence) was uniformly bounded in $B_8(0)$. Since $\Sigma_j$ converges to a hyperplane, the uniform curvature bounds and standard elliptic estimates conclude the proof.
\end{proof}

\begin{rema}
It is possible to conclude in a slightly different manner, by using the curvature estimates from Proposition \ref{prop:curv-est-sheeting-all-dim} with $\eta < 1$ to prove that the function $f_{j}(x) : = |x-z_{j}|^{2}$ is strictly convex (for any $j$ large enough). From this, if we assume sheeting of $\Sigma_j$ away from a point $y$ (second alternative of Theorem \ref{theo:sheeting-away-from-pt}), it is not hard to argue (using a max-min argument) that distinct sheets of $\Sigma_{j} \cap \left(B_9(0) \setminus B_{1/2}(y)\right)$ cannot be connected in $B_{1/2}(y)$ (endowing $\Sigma_j$ with the topology of the immersion, not of the embedding); therefore each connected component of $\Sigma_j$ in $B_9(0)$ contains exactly one sheet of $\Sigma_{j} \cap \left(B_9(0) \setminus B_{1/2}(y)\right)$. Theorem \ref{theo:main-weakly-st-sheeting} then follows from Allard's regularity theorem applied to each connected component individually, since standard arguments show that each component converges to the hyperplane with multiplicity $1$ in the sense of varifolds. This alternative argument seems to be necessary for the applications to bounded index surfaces mentioned in Section \ref{boundedindexminimal} (cf.\ \cite{CKM}).
\end{rema}

\appendix
\section{Weakly stable minimal hypersurfaces have only one end}
\label{appendixCCZ}

In this appendix we review a result of Cheng--Cheung--Zhou \cite{ChengCheungZhou} for weakly stable complete, non-compact minimal hypersurfaces immersed in ${\mathbb R}^{n+1}$ that generalized earlier work of Cao--Shen--Zhu \cite{CaoShenZhu} establishing the same result for strongly stable complete, non-compact minimal hypersurfaces. We include their proof here, since for our purposes we need to extend (as we do below) the argument to the case where the hypersurfaces are allowed to have a small singular set. We recall that we write $\nabla$ to denote the intrinsic gradient on a hypersurface, and will specify $\nabla^{\RR^{n+1}}$ if we refer to the ambient gradient.

\begin{theo}[{\cite[Theorem 3.2]{ChengCheungZhou}}]\label{theo:CCZ-smooth}
A complete connected oriented weakly stable minimal hypersurface immersed in $\RR^{n+1}$, $n\geq 2$, has only one end. 
\end{theo}
\begin{proof}
In $\RR^{3}$, a complete oriented weakly stable minimal surface is a plane, by \cite{DaSilveira}, so we need only consider $n\geq 3$. 

Suppose that $\Sigma$ is a complete oriented weakly stable {minimal} hypersurface in $\RR^{n+1}$ with at least two ends. By \cite[Lemma 2]{CaoShenZhu}, there exists a non-constant bounded harmonic function $u$ with finite Dirichlet energy, $\int_{\Sigma} |\nabla u|^{2} < \infty$. 

Consider $\varphi \in C^{1}_{0}(\Sigma)$ so that $\int_{\Sigma} \varphi|\nabla u| = 0$ (we will choose a specific $\varphi$ below). Plugging $\varphi|\nabla u|$ into the stability inequality for $\Sigma$ yields 
\begin{align*}
\int_{\Sigma} |A|^{2} \varphi^{2} |\nabla u|^{2} & \leq \int_{\Sigma} |\nabla(\varphi|\nabla u|)|^{2}\\
& = \int_{\Sigma} |\nabla |\nabla u||^{2}\varphi^{2} + \frac 12 \nabla |\nabla u|^{2} \cdot \nabla \varphi^{2} + |\nabla u|^{2}|\nabla \varphi|^{2}\\
& = \int_{\Sigma} |\nabla |\nabla u||^{2}\varphi^{2} - \frac 12 \Delta |\nabla u|^{2}  \varphi^{2} + |\nabla u|^{2}|\nabla \varphi|^{2}\\
& = \int_{\Sigma} ( |\nabla |\nabla u||^{2} -|D^{2}u|^{2}) \varphi^{2} - \Ric_{\Sigma}(\nabla u,\nabla u)\varphi^{2}+ |\nabla u|^{2}|\nabla \varphi|^{2},
\end{align*}
where we have integrated by parts and used the Bochner formula on $\Sigma$
\[
\frac 12 \Delta|\nabla u|^{2} = |D^{2}u|^{2} +\Ric_{\Sigma}(\nabla u,\nabla u).
\]
The Gauss equations (and minimality of $\Sigma$) imply that
\[
\Ric_{\Sigma}(\nabla u,\nabla u) = -|A(\nabla u,\cdot)|^{2} \geq - |A|^{2}|\nabla u|^{2}.
\]
Moreover, because $u$ is harmonic, we have the improved Kato inequality
\[
|D^{2}u|^{2} - |\nabla |\nabla u||^{2} \geq \frac{1}{n-1}|\nabla |\nabla u||^{2}.
\]
Combined with the stability inequality as above, this yields
\[
\int_{\Sigma} |\nabla|\nabla u||^{2} \varphi^{2} \leq (n-1) \int_{\Sigma} |\nabla u|^{2} |\nabla \varphi|^{2}.
\]

We now choose $\varphi$ appropriately. First, we argue that $\int_{\Sigma} |\nabla u| = \infty$. Let $p\in \Sigma$ such that $|\nabla \, u|(p) >0$. Write $B^{\Sigma}_{r}(p)$ for the intrinsic ball of radius $r$ around $p$ in $\Sigma$. For almost every $R\geq 1$, we have
\[
0 < \Vert \nabla u \Vert_{L^{2}(B_{1}^{\Sigma}(p))}^{2} \leq \Vert \nabla u \Vert_{L^{2}(B_{R}^{\Sigma}(p))}^{2} = \int_{B_{R}^{\Sigma}(p)} |\nabla u|^{2} = \int_{\partial B_{R}^{\Sigma}(p)} u \frac{\partial u}{\partial \nu} \leq \Vert u \Vert_{L^{\infty}} \int_{\partial B_{R}^{\Sigma}(p)} |\nabla u|,
\]
so because $u$ is bounded, there is some constant $C=C(u,p)>0$ so that 
\[
\int_{\partial B_{R}^{\Sigma}(p)} |\nabla u| \geq C > 0. 
\]
Hence, by the co-area formula
\[
\int_{B_{R}^{\Sigma}(p)\setminus B_{1}^{\Sigma}(p)} |\nabla u| \geq \int_{1}^{R} \int_{\partial B_{r}^{\Sigma}(p)} |\nabla u| \geq C(R-1),
\]
which tends to infinity as $R\to\infty$. 

For $a,b,R$ to be chosen with $R>a>0$, we define (for $t\in[0,1]$)
\[
\varphi_{t}(x) : = \begin{cases}
1 & d_{\Sigma}(x,p) < a\\
(a+R-d_{\Sigma}(x,p))/R & a \leq d_{\Sigma}(x,p) < a + R\\
t(a+R-d_{\Sigma}(x,p))/R & a + R \leq d_{\Sigma}(x,p) < a+2R\\
-t & a+2R \leq d_{\Sigma}(x,p) < a+2R+b\\
t(d_{\Sigma}(x,p)-(a+3R+b))/R & a+2R+b\leq d_{\Sigma}(x,p) < a+3R+b\\
0 & d_{\Sigma}(x,p) \geq a+3R+b.
\end{cases}
\]
For $\epsilon>0$ fixed, choose $R$ so that $(n-1)R^{-2}\int_{\Sigma} |\nabla u|^{2} < \epsilon $. Then,
\[
\int_{\Sigma} \varphi_{0} |\nabla u| \geq \int_{B_{a}^{\Sigma}(p)}|\nabla u| > 0
\]
and
\[
\int_{\Sigma} \varphi_{1} |\nabla u| \leq \int_{B_{a+R}^{\Sigma}(p)} |\nabla u| - \int_{B_{a+2R+b}^{\Sigma}(p) \setminus B_{a+2R}}|\nabla u|.
\]
Since we have seen that $\int_{\Sigma} |\nabla u| = \infty$, we may take $b=b(a,\epsilon)$ sufficiently large so that 
\[
\int_{\Sigma} \varphi_{1} |\nabla u|  < 0.
\]
Thus, there is some $t = t(b,\epsilon) \in (0,1)$ so that 
\[
\int_{\Sigma} \varphi_{t} |\nabla u| = 0.
\]

Choosing $\varphi_{t}$ in the above computation, we note that $|\nabla \varphi_{t}| \leq R^{-1}$, so
\[
\int_{\Sigma} |\nabla |\nabla u||^{2} \varphi_{t}^{2} \leq (n-1) \int_{\Sigma} |\nabla u|^{2} |\nabla \varphi_{t}|^{2} \leq (n-1)R^{-2} \int_{\Sigma} |\nabla u|^{2} < \epsilon.
\]
Thus, we find that 
\[
\int_{B_{a}^{\Sigma}(p)} |\nabla |\nabla u||^{2} < \epsilon. 
\]
Since $a$ and $\epsilon$ were arbitrary, we find that $\nabla |\nabla u| = 0$ along $\Sigma$, so $|\nabla u|$ is constant. Since $\Sigma$ has infinite volume and $|\nabla u| \in L^{2}(\Sigma)$, we find that $|\nabla u| = 0$. Thus, $u$ is constant, a contradiction. This completes the proof. 
\end{proof}

We now explain how the preceding argument generalizes to the case where the minimal hypersurface is allowed to have a small singular set. 
\begin{theo}\label{theo:CCZ-non-smooth}
For $n \geq 3$, suppose that $V$ is a stationary integral $n$-varifold in $\RR^{n+1}$ with ${\rm spt} \, \|V\|$ connected, ${\rm dim}_{\mathcal H} \, ({\rm sing} \, V) \leq n-7$ and with ${\rm sing} \, V \subset B_{1}(0)$. Assume that the regular part ${\rm spt} \, \|V\| \setminus {\rm sing} \, V$ is weakly stable. 
Then ${\rm spt} \, \|V\|$ has exactly one end at infinity. 
\end{theo}
\begin{proof}
Let $M = {\rm spt} \, \|V\| \setminus {\rm sing} \, V$ and suppose that $M$ has two (or more) ends at infinity. The proof proceeds as in the smooth case (Theorem \ref{theo:CCZ-smooth} above) with a few additional arguments. 
First we note that by \cite[Theorem A (ii)]{Ilmanen}, $\sing V$ does not disconnect $\supp \|V\|$, so $M$ is connected. This is needed at the end of the proof in order to have that the identical vanishing of $\nabla u$ implies the global constancy of $u$ (and not just the local constancy), which provides the desired contradiction.

The next argument concerns the existence of a non-constant bounded harmonic function $u$ with finite energy on $M$.  This is proved in the case $\sing V=\emptyset$ in \cite[Lemma 2]{CaoShenZhu} (and is used,  as noted, in Theorem \ref{theo:CCZ-smooth} above). The completeness assumption in \cite[Lemma 2]{CaoShenZhu} is not necessarily fulfilled by $M$, so it does not seem possible to simply invoke that result. We note, however, that completeness  is used in \cite{CaoShenZhu} only to infer that each end has infinite volume (\cite[Lemma 1]{CaoShenZhu}); this fact on the other hand follows directly from the monotonicity formula. Then we can follow verbatim the arguments in \cite[Lemma 2]{CaoShenZhu}, only with the following additional care. When we exhaust the hypersurface with domains $D_i$ we should remove, from the $D_i$ constructed in \cite{CaoShenZhu}, the closure of a smooth tubular neighbourhood of $\sing  V$ (whose size shrinks as $i\to \infty$). This will produce further boundary components, in addition to those in \cite{CaoShenZhu}, on which we will set boundary value $0$ when solving the Dirichlet problem \cite[(2)]{CaoShenZhu}. 

One more additional argument is needed in view of the fact that the test function $\varphi_t |\nabla u|$ constructed in Theorem \ref{theo:CCZ-smooth} might fail, a priori, to be an admissible function for the stability inequality. Indeed, $\varphi_t |\nabla u|$ is not compactly supported on $M$ and we do not have sufficient control of $|\nabla u|$ near $\sing  V.$ (If e.g.\ $|\nabla u|$ were bounded near ${\rm sing} \, V$, a straightforward capacity argument would suffice.) In order to overcome this difficulty,  we first observe an energy growth estimate for $u$ in balls centred on $\sing  V$ (inequality (\ref{eq:finerestimateonnablau}) below) which is obtained as follows: Since $\Delta u = 0$ on $M$, we see by integrating by parts and using the Cauchy--Schwarz inequality that for any $\phi \in C^1_c(M)$,
\[
\int_M |\nabla  u|^2 \phi^2 = -\int_M 2\phi u \nabla  u \nabla  \phi \leq 2 \left(\int_M \phi^2 |\nabla  u|^2\right)^{1/2} \left(\int_M   u^2  |\nabla  \phi|^2\right)^{1/2},
\]
from which we immediately get (using also the bounds $-1\leq u \leq 1$)
\[
\int_M |\nabla  u|^2 \phi^2 \leq 4\int_M   u^2  |\nabla  \phi|^2 \leq 4\int_M   |\nabla  \phi|^2 .
\]
Consequently, a standard capacity argument that only needs that the $2$-capacity of $\sing V$ is $0$ (true in view of $\mathcal{H}^{n-2}(\sing V)=0$) gives that the inequality 
\[
\int_M |\nabla  u|^2 \phi^2  \leq 4\int_M   |\nabla  \phi|^2 
\]
holds for all $\phi \in C^1_c(\RR^{n+1})$. In particular, choosing any $p\in \supp V$ and $\phi \in C^1_c(B_{2r}(p))$ to be a standard bump function that is identically equal to $1$ on $B_{r}(p)$ and identically equal to $0$ on the complement of $B_{2r}(p)$ with $|\nabla^{\RR^{n+1}}\phi|\leq \frac{2}{r}$, the preceding inequality and the monotonicity formula give
\begin{equation}
 \label{eq:finerestimateonnablau}
 \int_{B_{r}(p) \cap M} |\nabla  u|^2  \leq C r^{n-2}, 
 \end{equation}
where $C$ is independent of $r$.

With  this we proceed as follows: Let $\delta >0$ and let $\{B_{r_i}(x_i)\}_{i=1}^N$ be a cover of ${\rm sing} \, V$ (compact since ${\rm sing}\, V \subset B_{1}(0)$) with $\sum_{i=1}^{N}r_i^{n-4}\leq \delta$ (possible since ${\rm dim}_{\mathcal H} \, ({\rm sing} \, V) \leq n-7$). Defining a cutoff function $\zeta_{\delta}=\min_{i\in\{1, ..., N\}} \zeta_i$, where $\zeta_i \in C^{1}({\mathbb R}^{n+1})$ with $0 \leq \zeta_{i} \leq 1$, $\zeta_{i} = 0$ in $B_{r_{i}}(x_{i})$, $\zeta = 1$ in ${\mathbb R}^{n+1} \setminus B_{2r_{i}}(x_{i})$ and $|D\zeta_{i}| \leq 2r_{i}^{-1}$, in view of (\ref{eq:finerestimateonnablau}) we see that  
\[
\int_{M} |\nabla \zeta_\delta|^2 |\nabla u|^2\leq 2\sum_{i=1}^N \int_{B_{2r_i}} |\nabla \zeta_i|^2 |\nabla u|^2 \leq 8\sum_{i=1}^N r_i^{-2} \int_{B_{2r_i}} |\nabla u|^2 \leq 8C \sum_{i=1}^Nr_i^{n-4} \leq 8C \delta
\]
whence 
\begin{equation}
 \label{eq:controlonzetadelta}
\int_{M} |\nabla \zeta_\delta|^2 |\nabla u|^2\to 0
\end{equation}
as $\delta \to 0.$ 
We can now adapt the arguments in Theorem~\ref{theo:CCZ-smooth}: let $\epsilon>0$ be arbitrary and choose $R>0$ so that $(n-1)R^{-2}\int_{\Sigma} |\nabla u|^{2} < \epsilon $. For every $\delta>0$ choose a compactly supported function $\varphi_{t,\delta}$ that is constructed in the manner $\varphi_{t}$ is constructed in Theorem \ref{theo:CCZ-smooth} so as to ensure $\int\zeta_\delta \varphi_{t,\delta} |\nabla u|=0$ (one can work with the function $\varphi_t$ defined above and set $a\geq 1$ so that $\varphi_t$ is $1$ on the singular set and $\zeta_\delta |\nabla u|=|\nabla u|$ on $B_{a+2R+b}^{\Sigma}(p) \setminus B_{a+2R}^{\Sigma}(p)$; the $t$ for which the zero-average condition is met will depend however on $\delta$, hence the dependence of $\varphi_{t,\delta}$ on $\delta$). As $\delta \to 0$, since $|\varphi_{t,\delta}|\leq 1$ and $|\nabla \varphi_{t,\delta}|\leq \frac{1}{R}$ are uniformly bounded, and moreover $b(\epsilon, a)$ can be chosen independently of $\delta$, we get  for a sequence $\delta_{j} \to 0^{+}$ that $\varphi_{t,\delta_{j}} \to \tilde{\varphi}$ for a compactly supported Lipschitz function satisfying $\int \tilde{\varphi}  |\nabla u|=0$ and $\tilde{\varphi}$ identically $1$ on $B_a^\Sigma(p)$.

Plugging the (admissible) test function $\zeta_\delta \varphi_{t,\delta} |\nabla u|$ in the stability inequality we get, arguing as in the proof of Theorem \ref{theo:CCZ-smooth} by means of Bochner's formula and Gauss equations,
\begin{align*}
& \int_{M} |\nabla|\nabla u||^{2} (\zeta_\delta \varphi_{t,\delta})^{2}\\
 & \leq (n-1) \int_{M} |\nabla u|^{2} |\nabla (\zeta_\delta \varphi_{t,\delta})|^{2}\\
& =(n-1)   \int_{M} |\nabla \zeta_\delta|^2 \varphi_{t,\delta}^2 |\nabla u|^2 + 2  \varphi_{t,\delta} \zeta_\delta |\nabla u|^2 \nabla \zeta_\delta\cdot \nabla\varphi_{t,\delta}+\zeta_\delta^2 |\nabla \varphi_{t,\delta}|^2 |\nabla u|^2 ;
\end{align*}
using the Cauchy--Schwarz inequality for the middle term on the right-hand side, setting $\delta = \delta_{j}$ and letting $j \to \infty$ we obtain (recalling (\ref{eq:controlonzetadelta})) that the first and second term on the right-hand side vanish in the limit; moreover, with the choices of $\epsilon$ and $R$ recalled above, using that $\zeta_{\delta_{j}} \uparrow 1$, $\varphi_{t,\delta_{j}}\to \tilde{\varphi}$ as $\delta_{j} \to 0$, and that $|\nabla \varphi_{t,\delta_{j}}|\leq \frac{1}{R}$, we get 
\[
\int_{M} |\nabla|\nabla u||^{2}  \tilde{\varphi}^{2} \leq \epsilon,
\]
which allows us to conclude, in view of the arbitrariness of $\epsilon$, that $|\nabla u|=0$ on $\Sigma$ and obtain the desired contradiction with the non-constancy of $u$.
\end{proof}

\section{Results from \cite{BW:stableCMC}}
\label{appendixBW}

We recall here the regularity/compactness results from \cite{BW:stableCMC} that are used in the proof of Theorem \ref{theo:main-weakly-st-sheeting}. The first one is a combination of \cite[Theorem 2.1]{BW:stableCMC} and \cite[Theorem 2.3]{BW:stableCMC}.

\begin{theo}[\textbf{regularity/compactness for weakly stable CMC hypersurfaces}]
\label{theo:BWregcomp}
Let $n \geq 2$, $R$, $K_0$, $H_0 \in (0,\infty)$ be fixed. Denote by $\mathcal{S}_{H_0, K_0}(B_R^{n+1}(0))$ the class of all hypersurfaces $M$ in $B_R^{n+1}(0)$ such that

\begin{itemize}
 \item $M$ is an immersed, smooth, weakly stable, CMC hypersurface (not necessarily complete) in $B_R^{n+1}(0)$, with integer multiplicity (constant on every connected component of the immersion);
 
 \item $\mathcal{H}^{n-7+\alpha}\left(\overline{M}\setminus M\right)=0$ for all $\alpha>0$ (i.e.~$M$ is allowed to have a singular set of co-dimension at least $7$);

 \item $M$ has no transverse points; equivalently (by the strong maximum principle), at every $p \in M$ where $M$ is not embedded, there exists $\rho>0$ such that $M \cap B^{n+1}_\rho(p)$ is the union of exactly two embedded complete smooth CMC hypersurfaces in $B^{n+1}_\rho(p)$ that intersect only tangentially;
 
 \item the modulus $H$ of the mean curvature of $M$ is $\leq H_0$; 
 
 \item $\mathcal{H}^n(M) \leq K_0$.
\end{itemize}

Then $\mathcal{S}_{H_0, K_0}(B_R^{n+1}(0))$ is a compact family in the varifold topology. Moreover, if $V_n \in \mathcal{S}_{H_0, K_0}$ and $V_n \rightharpoonup V$ the (constant) mean curvature of $V$ is given by $\lim_{n\to \infty} H_n$, where $H_n$ is the (constant) mean curvature of $V_n$.
\end{theo}

The next result is a synthesis of \cite[Theorem 2.1]{BW:stableCMC}, \cite[Theorem 3.1]{BW:stableCMC}, \cite[Theorem 3.3]{BW:stableCMC} and \cite[Lemma 8.1]{BW:stableCMC}.

\begin{theo}[\textbf{sheeting away from a point for weakly stable CMC hypersurfaces}]
\label{theo:sheeting-away-from-pt}
Let $V_j \to V$, where $V_j \in \mathcal{S}_{H_0, K_0}(B_R^{n+1}(0))$ (with the notations from the previous statement) and $V$ is a sum of parallel hyperplanes, each with a constant integer multiplicity.
Then, up to a rotation of coordinates
 \begin{itemize}
  \item either, for every $k$ large enough, we have
  
  $$\supp{V_k} \text{ restricted to } \left(B_{\frac{R}{2}}^{n}(0) \times {\RR} \right) = \cup_{j=1}^{q} {\rm graph} \, u_{j},$$ where $u_{j} \in C^{2, \alpha} \, (B_{\frac{R}{2}}^{n}(0); {\mathbb R})$, $u_j$ are separately smooth CMC graphs (possibly with tangential intersections) with small gradients and $u_{1} \leq u_{2} \leq \ldots \leq u_{q}$,
  
  \item or there exists a point $y \in \supp{\|V\|} \cap  \left(\overline{B}_{R/2}^{n}(0) \times {\RR} \right)$ and a subsequence $V_{j'}$ such that, for any $r>0$, the following holds: for $j'$ large enough (depending on $r$) $V_{j'}$ is strongly stable in $\left(B_{R}^{n}(0) \times {\RR} \right) \setminus B^{n+1}_r(y)$ and moreover $V_{j'}$ converges smoothly (with sheeting and possibly with multiplicity) to $V$ away from $y$, in the following sense. With the notation $V=\sum q_i |W_i|$ (where each $W_i$ is one of the parallel hyperplanes of $\text{supp}\,V$), $\sum_i q_i =q$ and $y\in W_1$, if $r<\text{dist}(W_1, W_i)$ for $i\neq 1$, then for $j'$ large enough (depending on $r$), the following decomposition holds:

$$\supp{V_{j'}} \text{ restricted to } \left(B_{\frac{9R}{10}}^{n}(0) \times {\RR} \right) \setminus B^{n+1}_r(y) = \cup_{j=1}^{q_1} {\rm graph} \, u_{j} \bigcup \cup_{j=q_1+1}^{q} {\rm graph} \, u_{j}$$
where $u_{j} \in C^{2, \alpha} \, (B_{\frac{9R}{10}}^{n}(0)\setminus B^n_r(y); {\mathbb R})$ for $j=1, ..., q_1$, $u_{j} \in C^{2, \alpha} \, (B_{\frac{9R}{10}}^{n}(0); {\mathbb R})$ for $j=q_1+1, ..., q$ and the $u_j$ are separately smooth CMC graphs (possibly with tangential intersections) with small gradients and $u_{1} \leq u_{2} \leq \ldots \leq u_{q}$.
 \end{itemize}
 
\end{theo}

\bibliography{bib} 
\bibliographystyle{amsalpha}

\end{document}